\documentclass[11pt,a4paper]{article}

\usepackage[utf8]{inputenc}
\usepackage{amsmath}
\usepackage{amsfonts}
\usepackage{amssymb}
\usepackage{amsthm}
\usepackage{bbm}
\usepackage{xcolor}
\usepackage{graphicx}
\usepackage{hyperref}
\numberwithin{equation}{section}
\newtheorem*{theorem*}{Theorem}

\newtheorem{theorem}{Theorem}[section]

\newtheorem{proposition}[theorem]{Proposition}

\newtheorem{remark}[theorem]{Remark}

\DeclareMathOperator{\sign}{sign}

\title{Low-dimensional Cox-Ingersoll-Ross process}

\author{Yuliya Mishura$^{1,2}$ \\ \href{mailto:yuliyamishura@knu.ua}{yuliyamishura@knu.ua}
    \and Andrey Pilipenko$^{3,4}$\\\href{mailto:pilipenko.ay@gmail.com}{pilipenko.ay@gmail.com}  
    \and Anton Yurchenko-Tytarenko$^5$ \\ \href{mailto:antony@math.uio.no}{antony@math.uio.no}}

\date{%
    $^1$Department of Probability, Statistics and Actuarial Mathematics, Taras Shevchenko National University of Kyiv\\%
    $^2$Division of Mathematics and Physics, M\"alardalen University\\
    $^3$Institute of Mathematics, National Academy of Sciences of Ukraine\\
    $^4$Igor Sikorsky Kyiv Polytechnic Institute \\
    $^5$Department of Mathematics, University of Oslo
}

\begin{document}

\maketitle

\begin{abstract}
    The present paper investigates Cox-Ingersoll-Ross (CIR) processes of dimension less than 1, with a focus on obtaining an equation of a new type including local times for the square root of the CIR process. We utilize the fact that non-negative diffusion processes can be obtained by the transformation of time and scale of some reflected Brownian motion to derive this equation, which contains a term characterized by the local time of the corresponding reflected Brownian motion. Additionally, we establish a new connection between low-dimensional CIR processes and reflected Ornstein-Uhlenbeck (ROU) processes, providing a new representation of Skorokhod reflection functions.
\end{abstract}

\section{Introduction}

\subsection{Background and motivation}

The squared Bessel process
\begin{equation}\label{eq: Bessel introduction}
    X(t) = x_0 + at + 2\int_0^t \sqrt{X(s)}dW(s)
\end{equation}
as well as its generalization Cox-Ingersoll-Ross (CIR) process 
\begin{equation}\label{0.q}
   X(t) = x_0 + \int_0^t \left(a - bX(s)\right)ds + \sigma \int_0^t \sqrt{X(s)}  dW(s)
\end{equation}
where $x_0 \ge 0$, $a>0$, $b\in\mathbb R$, and their respective square roots are widely used in various fields, in particular physics (see e.g. \cite{Bray2000, Horibe_Hosoya_Sakamoto_1983} and the overview in \cite[Section I]{Guarnieri_Moon_Wettlaufer_2017}) and finance \cite{CIR1981, CIR1985-1, CIR1985-2, Heston_1993}. One of the reasons for the popularity of these processes lies in the well-known fact (see e.g. \cite[Chapter IV, Example 8.2]{IW1977}) that $a > 0$ in \eqref{0.q} implies that $X(t) \ge 0$ for all $t\ge 0$ with probability 1, which is a natural property for multiple real-life phenomena. Furthermore, if the \textit{Feller condition} $2a \ge \sigma^2$ is satisfied, the paths of $X$ in \eqref{0.q} are strictly positive a.s., which turns out to be very useful in multiple cases. For example, the well-known Heston model \cite{Heston_1993} utilizes $Y := \sqrt{X}$ as stochastic volatility and, under the Feller condition,
$Y$ has the dynamics of the form
\begin{equation}\label{eq: Y equation for big a introduction}
    Y(t) = \sqrt{x_0} + \frac{1}{2}\int_0^t \left( \frac{a-\frac{\sigma^2}{4}}{Y(s)} - bY(s) \right)ds + \frac{\sigma}{2} W(t), \quad t\ge 0,
\end{equation}
since it is evident that
\begin{equation}\label{eq: 1 over Y introduction}
    \int_0^t \frac{1}{Y(s)}ds < \infty
\end{equation}
with probability 1 for all $t\ge 0$. This equation can be used for e.g. simulation purposes (see, for example, \cite{Alfonsi_2013, Dereich_Neuenkirch_Szpruch_2012, Neuenkirch_Szpruch_2014}); moreover, the measure change procedure associated with the Heston model naturally involves the inverse volatility $1/Y$ which has far more transparent properties when $X>0$ a.s.

At the same time, empirical considerations indicate that the Feller condition $2a \ge \sigma ^2$ can sometimes be too restrictive and models perform better when it is not satisfied. For instance, \cite[Section 3.4]{Kokholm_Stisen_2015} reports that the joint SPX-VIX fit of the Heston model turns out to be substantially better when the Feller condition is not demanded from the model parameters. Additionally, \cite[Example 10.2.6]{Alos_Garcia_Lorite_2021} indicates that the Heston model with violated Feller condition can reproduce the upward VIX ``smirk''. In other words, there are cases when the process $Y = \sqrt{X}$ under relatively small values of $a$ turns out to be more relevant for reflecting real-life phenomena despite the associated analytical challenges. Nevertheless, the majority of sources in the literature pay more attention to the case when the Feller condition is satisfied. Among notable exceptions, we mention \cite{AryasovaPilipenko_strong, Blei_2012, Cherny_2000, Cherny_Engelbert_2005} which discussed the SDEs of the type \eqref{eq: Y equation for big a introduction} when $\frac{\sigma^2}{4} < a < \frac{\sigma^2}{2}$. It is worth to note a more recent paper \cite{MYT2022} which establishes a connection between $Y = \sqrt{X}$ and a reflected Ornstein-Uhlenbeck (ROU) process 
\[
    Y_0(t) = \sqrt{x_0} - \frac{b}{2} \int_0^t Y_0(s) ds + \frac{\sigma}{2}W(t) + L_0(t), 
\]
where $L_0$ is the corresponding Skorokhod reflection function, i.e. a continuous non-decreasing process that has points of growth exclusively when $Y_0(t)=0$ and such that $Y_0(t) \ge 0$. In particular, it is established that $Y_0 = \sqrt{X}$ when $a=\frac{\sigma^2}{4}$ in \eqref{0.q}. Additionally, \cite[Theorem 2.4]{MYT2022} provides a new representation of $L_0$ in terms of a limit of the CIR processes: with probability 1, for any positive sequence $\{\varepsilon_n,~n\ge 1\}$ such that $\varepsilon_n \downarrow 0$, $n\to\infty$, and for all $T>0$
\begin{equation}\label{eq: uniform convergence 2, standard case}
    \sup_{t\in[0,T]} \left|L_0(t) - \frac{1}{2}\int_0^t \frac{\varepsilon_n}{\sqrt{X_{\varepsilon_n}(s)}}ds\right| \to 0, \quad n\to\infty,
\end{equation}
where
\[
    X_{\varepsilon_n}(t) = x_0 + \int_0^t \left(\frac{\sigma^2}{4} + \varepsilon_n - bX_{\varepsilon_n}(s)  \right) ds + \sigma\int_0^t \sqrt{X_{\varepsilon_n}(s)}dW(s).
\]

The representation of $L_0$ from \cite{MYT2022} described above essentially concerns convergence of the CIR square roots as $a\to \frac{\sigma^2}{4} +$ and does not cover what happens when $a\to\frac{\sigma^2}{4}-$. The reason is that analytic challenges associated to the process $Y = \sqrt{X}$ are especially acute when $0 < a < \frac{\sigma^2}{4}$, i.e. when the \textit{dimension} (see e.g. \cite{Maghsoodi1996}) $k:=\frac{4a}{\sigma^2}$ of the process \eqref{0.q} is less than 1. Indeed, the integral in \eqref{eq: 1 over Y introduction} is infinite after the first moment of hitting zero, the representation \eqref{eq: Y equation for big a introduction} does not hold and, furthermore, the process $Y = \sqrt{X}$ \textit{is not a semimartingale} (see e.g. Example 1.2 and Appendix 1 in \cite{Mijatovic_Urusov_2015} or \cite[p. 100]{Graversen_Peskir_1998}). In this regard, one must mention important contributions \cite{Bertoin_1990-1, Bertoin_1990b} which shed light on the behavior of $Y = \sqrt{X}$ when $X$ is the squared Bessel process \eqref{eq: Bessel introduction} of dimension $k = a \in (0,1)$. There, it is shown that $Y$ satisfies the equation of the form
\begin{equation}\label{eq: Y Bertoin intro}
    Y(t) = \sqrt{x_0} + W(t) + L(t),
\end{equation}
where 
\begin{equation}\label{eq: L Bertoin intro}
   L(t) := \frac{a-1}{2}\int_0^\infty y^{a-2} (\ell(t,y) - \ell(t,0))dy    
\end{equation}
with $\ell$ being a jointly continuous in $(t,y)$ normalized local time such that for any bounded measurable function $f$
\begin{equation}\label{eq: l Bertoin intro}
    \int_0^t f(Y(s))ds = \int_0^\infty f(y) y^{a-1} \ell(t,y)dy.   
\end{equation}

\subsection{Main results}

In our paper, we consider a more general case of the CIR process \eqref{0.q} with $b\in \mathbb R$ and $0 < a < \frac{\sigma^2}{4}$ (we call such a process a \textit{low-dimensional} CIR) and study the properties of $Y = \sqrt{X}$. More precisely, we represent $Y$ as a transformation of a reflected Brownian motion  $\widetilde W$ and use the properties of the local time $L^{\widetilde W}$ of the latter to study the local time $L^Y$ of $Y$. Afterwards, we use the connection between $L^{\widetilde W}$ and $L^Y$ to get a representation in the spirit of \eqref{eq: Y Bertoin intro}: namely, we prove that $Y$ is a strong solution of the equation
\begin{equation}\label{eq: eq for Y introduction}
    Y(t) = \sqrt{x_0} - \frac{b}{2}\int_0^t Y(s) ds + \frac{\sigma}{2} W(t) + L(t), 
\end{equation}
where
\begin{equation}\label{intro: L as local time}
    L(t) = -\frac{1}{2}\left(\frac{\sigma^2}{4} - a \right) \int_0^{\infty} y^{\frac{4a}{\sigma^2} - 2} \left( \ell(t,y) - \ell(t,0) \right) dy
\end{equation}
with $\ell$ being an explicitly given normalized transformation of a local time $L^Y$ of the process $Y$:
\[
\ell(t,y):=y^{1-\frac{4a}{\sigma^2}} L^Y(t,y),
\]
where $\ell(t,0):=\lim_{y\to0+}\ell(t,y)$ is defined by continuity.

Finally, we close the gap of \cite{MYT2022} mentioned above and obtain a representation of the Skorokhod reflection function for the ROU process in terms of CIR processes of dimension $k = \frac{4a}{\sigma^2} < 1$.

It is worth noting that our approach is simpler than the one in \cite{Bertoin_1990-1, Bertoin_1990b} and is based on the following machinery: we notice that It\^o's formula applied to $\sqrt{X(t) + \varepsilon}$ followed by moving $\varepsilon \downarrow 0$ implies that $Y = \sqrt{X}$ satisfies the equation of the form \eqref{eq: eq for Y introduction} with $L$ represented as an a.s.-limit
\begin{equation}\label{eq: L as limit intro}
    L(t) := \lim_{n\to\infty} \frac{1}{2} \int_0^t \left( \frac{a}{\sqrt{X(s) + \varepsilon_n}} - \frac{\sigma^2}{4} \frac{X(s)}{(X(s) + \varepsilon_n)^{\frac{3}{2}}} \right)ds
\end{equation}
and $\{\varepsilon_n,~n\ge 0\}$ being some sequence converging to zero. After that, we utilize the fact that the CIR process $X$ is a regular diffusion and hence can be obtained from some reflected Brownian motion $\widetilde W$ by a transformation of time and scale (see e.g. \cite[Chapter V, Section 7]{Rogers_Williams_1987}). We find the explicit shape of this transformation, use it to establish the connection between the local times of $\widetilde W$ and $Y$. Finally, we exploit this link to show that the limit \eqref{eq: L as limit intro} is equal to \eqref{intro: L as local time}. The technique described above seems to be more transparent than the one employed in \cite{Bertoin_1990-1, Bertoin_1990b} and additionally allows to get a clear intuition behind the process $\ell$ in \eqref{eq: l Bertoin intro}.

\subsection{Structure of the paper}

The paper is organized as follows. In Section \ref{sec: preliminaries}, we present some preliminary calculations and discuss the representation \eqref{eq: eq for Y introduction}--\eqref{eq: L as limit intro}. Section \ref{sec: hard case} is devoted to the case $0< a < \frac{\sigma^2}{4}$ and contains Theorem \ref{th: main theorem} that can be regarded as the main result of the paper. In Section \ref{sec: simple cases}, we discuss the results and compare them with the behavior of the limit in \eqref{eq: L as limit intro} when $a\ge \frac{\sigma^2}{4}$.  In Section~\ref{sec: connection to Skorokhod reflections}, we establish a new connection between CIR processes of dimension $k = \frac{4a}{\sigma^2} < 1$ and ROU processes and obtain a new representation of Skorokhod reflection function.

\section{Preliminary calculations}\label{sec: preliminaries}

Let $a,\sigma > 0$, $b\ge 0$, $W = \{W(t),~t\ge 0\}$ be a Brownian motion, and let us consider the continuous modification of a standard CIR process \eqref{0.q}. Note that, by \cite[Chapter IV, Example 8.2]{IW}, the paths of $X$ are non-negative with probability 1 provided that $a>0$ and hence one can define the square-root process $Y = \{Y(t),~t\ge 0\} := \{\sqrt{X(t)},~t\ge 0\}$. In order to analyze the dynamics of $Y$, take $\varepsilon > 0$ and observe that, by It\^o's formula,
\begin{equation}\label{eq: Ito}
\begin{aligned}
    \sqrt{X(t) + \varepsilon} &= \sqrt{x_0 + \varepsilon} + \frac{1}{2} \int_0^t \left( \frac{a}{\sqrt{X(s) + \varepsilon}} - \frac{\sigma^2}{4} \frac{X(s)}{(X(s) + \varepsilon)^{\frac{3}{2}}} \right)ds 
    \\
    &\quad - \frac{1}{2} \int_0^t \frac{b X(s)}{\sqrt{X(s) + \varepsilon}} ds + \frac{\sigma}{2} \int_0^t \frac{\sqrt{X(s)}}{\sqrt{X(s) + \varepsilon}} dW(s).
\end{aligned}
\end{equation}
Fix an arbitrary $T>0$ and note that the left-hand side of \eqref{eq: Ito} converges to $Y(t)$ uniformly on $[0,T]$ with probability 1 as $\varepsilon \downarrow 0$. It is also evident that 
\begin{equation}\label{eq: convergence CIR 2}
\begin{gathered}
    \sup_{t\in[0,T]}\left|\int_0^t \frac{X(s)}{\sqrt{X(s) + \varepsilon}} ds - \int_0^t Y(s) ds\right| \to 0 \quad a.s., \qquad \varepsilon \downarrow 0.
\end{gathered}
\end{equation}
Next, by \cite[Chapter XI]{Revuz_Yor_1999}, $\mathbb{E}\int_0^\infty  \mathbbm{1}_{\{X(s) = 0 \}} ds = 0$ and hence, by the Burkholder-Davis-Gundy inequality, for any $T>0$
\begin{equation*}
\begin{aligned}
    \mathbb E&\left(\sup_{t\in [0,T]}\left|\int_0^t \frac{\sqrt{X(s)}}{\sqrt{X(s) + \varepsilon}} dW(s) - W(t)\right|\right)^{2} 
    \\
    &\le 4\mathbb E \int_0^T \left(\frac{\sqrt{X(s)}}{\sqrt{X(s) + \varepsilon}} -1\right)^{2} ds
    \\
    &= 4\mathbb E\int_0^T \left(\frac{\sqrt{X(s)}}{\sqrt{X(s) + \varepsilon}} -1\right)^{2} \mathbbm{1}_{\{X(s) > 0 \}} ds + 4 \mathbb{E}\int_0^T  \mathbbm{1}_{\{X(s) = 0 \}} ds
    \\
    &= 4\mathbb E\int_0^T \left(\frac{\sqrt{X(s)}}{\sqrt{X(s) + \varepsilon}} -1\right)^{2} \mathbbm{1}_{\{X(s) > 0 \}} ds \to 0, \quad \varepsilon \downarrow 0.
\end{aligned}
\end{equation*}
This implies that for each $T>0$
\begin{equation}\label{eq: UCP to W}
    \sup_{t\in[0,T]}\left|\int_0^t \frac{\sqrt{X(s)}}{\sqrt{X(s) + \varepsilon_{n}}} dW(s) - W(t)\right| \xrightarrow{\mathbb P} 0, \quad \varepsilon \downarrow 0.
\end{equation}
In particular, \eqref{eq: Ito} as well as convergences \eqref{eq: convergence CIR 2} and \eqref{eq: UCP to W} imply that the left-hand side of
\begin{align*}
    \sqrt{X(t) + \varepsilon}& - \sqrt{x_0 + \varepsilon} + \frac{1}{2} \int_0^t \frac{b X(s)}{\sqrt{X(s) + \varepsilon}} ds - \frac{\sigma}{2} \int_0^t \frac{\sqrt{X(s)}}{\sqrt{X(s) + \varepsilon}} dW(s)
    \\
    & = \frac{1}{2} \int_0^t \left( \frac{a}{\sqrt{X(s) + \varepsilon}} - \frac{\sigma^2}{4} \frac{X(s)}{(X(s) + \varepsilon)^{\frac{3}{2}}} \right)ds 
\end{align*}
converges uniformly on compacts in probability as $\varepsilon \downarrow 0$. Therefore, there exists a ucp-limit
\begin{equation}\label{eq: def of L}
    L(t) := \lim_{\varepsilon \downarrow 0} \frac{1}{2} \int_0^t \left( \frac{a}{\sqrt{X(s) + \varepsilon}} - \frac{\sigma^2}{4} \frac{X(s)}{(X(s) + \varepsilon)^{\frac{3}{2}}} \right)ds
\end{equation}
and the process $Y = \sqrt{X}$ satisfies the SDE of the form
\[
    Y(t) = Y(0) - \frac{b}{2}\int_0^t Y(s) ds + \frac{\sigma}{2} W(t) + L(t),  
\]
where $Y(0) = \sqrt{x_0}$.

\begin{remark}\label{rem: anticipation of a.s. convergence}
    Ucp convergence in \eqref{eq: def of L} implies that for an arbitrary $T>0$ there exists a sequence $\{\varepsilon_{n},~n \ge 1\}$ (depending on $T$) such that, for any $t\in[0,T]$, 
    \begin{equation}\label{eq: a.s.}
    \begin{gathered}
        \left|L(t) - \frac{1}{2} \int_0^t \left( \frac{a}{\sqrt{X(s) + \varepsilon_n}} - \frac{\sigma^2}{4} \frac{X(s)}{(X(s) + \varepsilon_n)^{\frac{3}{2}}} \right)ds\right| \to 0
    \end{gathered}
    \end{equation}
    with probability~1 as $n\to \infty$. Later on, we will see that the a.s. convergence \eqref{eq: a.s.} holds for an \emph{arbitrary} sequence $\{\varepsilon_n,~n\ge 1\}$ such that $\varepsilon_n \downarrow 0$ as $n\to\infty$. Moreover, it will be shown that the set of full probability where \eqref{eq: a.s.} holds can be chosen independently of a particular sequence $\{\varepsilon_n,~n\ge 1\}$. 
\end{remark}

\begin{remark}
    Note that the process $L$ defined by \eqref{eq: def of L} is continuous a.s. since
    \[
        L(t) = Y(t) - Y(0) + \frac{b}{2}\int_0^t Y(s) ds - \frac{\sigma}{2} W(t).
    \]
\end{remark}

\section{Stochastic representation of $L$ when ${0 < a < \frac{\sigma^2}{4}}$}\label{sec: hard case}

Our strategy for the analysis of $L$ will be as follows. Since the CIR process $X$ in \eqref{0.q} is a non-negative regular diffusion, it can be represented (see e.g. \cite[Chapter V, Section 7]{Rogers_Williams_1987}) in the form
\[
    X(t) = S^{-1}\left(\widetilde W_{\tau_t}\right)
\]
for some change of time $\tau$ and change of scale $S$ of a reflected Brownian motion $\widetilde W = \{\widetilde W(t),~t\ge 0\}$. Then, we re-write the integral in the limit \eqref{eq: def of L} in terms of the local time $L^{\widetilde W}  = L^{\widetilde W}(t,x)$ of $\widetilde W$ and exploit H\"older continuity of the latter to find an explicit representation of $L$ in terms of $L^{\widetilde W}$.

\subsection{CIR process as the transformation of a reflected Brownian motion}

In order to implement our approach, we first need to represent the CIR process as a transformation of a reflected Brownian motion. For a given set of parameters $a$, $b$, $\sigma$ of the SDE \eqref{0.q}, define a \emph{scale function} $S$: $[0,\infty)\to[0,\infty)$ by
\begin{equation}\label{eq: S}
    S(x) := \int_0^x y^{-\frac{2a}{\sigma^2}} e^{\frac{2by}{\sigma^2}} dy
\end{equation}
and observe that, since $S$ is strictly increasing and $S(\infty) = \infty$, there exists its inverse $S^{-1}$. Define also a \emph{speed measure}
\[
    m(dx) = \rho(x)dx,
\]
where
\begin{equation}\label{eq: rho}
    \rho(x) := \frac{1}{\sigma^2} x^{\frac{4a}{\sigma^2}-1} e^{-\frac{4b}{\sigma^2} x} \mathbbm 1_{x>0}.
\end{equation}

\begin{proposition}\label{prop: CIR as RBM}
    Let $X$ be the unique strong solution to the CIR equation \eqref{0.q} with $0< a < \frac{\sigma^2}{4}$. Then there exists a reflected Brownian motion $\widetilde W$ starting at $S(x_0)$ such that 
    \begin{equation}\label{eq: X as time-changed Brownian motion}
        X(t) := S^{-1}\left( \widetilde W(\tau_t) \right),      
    \end{equation}
    where
    \begin{equation}\label{eq: tau}
        \tau_t := \varphi^{-1}_t
    \end{equation}
    with
    \begin{equation}\label{eq: phi}
        \varphi_t := \int_0^t \rho\left(S^{-1}\left(\widetilde W(s)\right)\right)ds.
    \end{equation}
\end{proposition}

Before moving to the proof of Proposition \ref{prop: CIR as RBM}, let us make some remarks regarding its formulation.

\begin{remark}\hfill
    \begin{enumerate}
        \item The process $\varphi$ in \eqref{eq: phi} is well-defined. Indeed, let $L^{\widetilde W} = \{L^{\widetilde W}(t,x),~t\ge 0,~x\ge 0\}$ be the local time of $\widetilde W$, i.e. for any bounded measurable $f$,
        \[
            \int_0^t f(\widetilde W(s))ds = \int_0^\infty f(x)L^{\widetilde W}(t,x)dx \quad a.s.
        \]
        Then, with probability 1,
        \begin{align*}
            \varphi_t &= \int_0^t \rho\left(S^{-1}\left(\widetilde W(s)\right)\right)ds
            \\
            & = \int_0^\infty \rho\left(S^{-1}\left(y\right)\right)L^{\widetilde W}(t,y)dy
            \\
            & = \int_0^\infty \rho\left(x\right)S'(x) L^{\widetilde W}(t,S(x))dx
            \\
            & = \frac{1}{\sigma^2}  \int_0^\infty x^{\frac{2a}{\sigma^2}-1} e^{-\frac{2b}{\sigma^2} x}  L^{\widetilde W}(t,S(x))dx
            \\
            & < \infty
        \end{align*}
        because $\frac{2a}{\sigma^2}-1 > -1$ and $L^{\widetilde W}(t,S(x)) = 0$ for $x > S^{-1} \left(\max_{s\in[0,T]} \widetilde W(s)\right)$.

        \item Since $\varphi$ is strictly increasing and $\varphi_\infty = \infty$ with probability 1, its inverse $\tau$ in \eqref{eq: tau} is well-defined a.s.
    \end{enumerate}
\end{remark}

\begin{remark}\label{rem: inverse transformation}
    The transformation \eqref{eq: X as time-changed Brownian motion} is invertible. Indeed, it is straightforward to check that, with probability 1,
    \[
        \tau_t = \int_0^t \frac{1}{\rho\left(X(s)\right)}ds,
    \]
    therefore,
    \begin{equation*}
        \widetilde W(t) = S\left( X\left( \varphi_t \right) \right),
    \end{equation*}
    where $\varphi = \tau^{-1}$ can be expressed as the inverse of the mapping $t \mapsto \int_0^t \frac{1}{\rho\left(X(s)\right)}ds$. For more details on transformations of this type, we refer the reader to \cite[Chapter IV, \S7]{IW}.
\end{remark}

\begin{proof}[Proof of Proposition \ref{prop: CIR as RBM}]
    We will split the proof into two steps. First, we will follow \cite[Chapter V, Section 7, \S 48]{Rogers_Williams_1987} to prove that, for \emph{some given reflected Brownian motion $\widetilde W$}, the process $X(t) := S^{-1}\left( \widetilde W(\tau_t) \right)$ is the weak solution to the SDE \eqref{0.q}. Then we will utilize the invertability of transformation \eqref{eq: X as time-changed Brownian motion} outlined in Remark \ref{rem: inverse transformation} to establish the existence of a reflected Brownian motion $\widetilde W$ together with the required representation for the given CIR process $X$.

    \textbf{Step 1.} Let $\widetilde Z = \{\widetilde Z(t),~t\ge 0\}$ be a standard Brownian motion starting at $\widetilde Z(0) = S(x_0)$. Consider a reflected Brownian motion
    \[
        \widetilde W(t) := |\widetilde Z(t)| = S(x_0) + Z(t) + L^{\widetilde Z}(t),
    \]
    where $Z(t) := \int_0^t \sign{\widetilde Z(s)} d \widetilde Z(s)$ is a Brownian motion and $L^{\widetilde Z}(t)$ is the local time of $L^{\widetilde Z}$ at zero. Put $V(t) := S^{-1}\left(\widetilde W(t)\right)$ and observe that, by the extension of It\^o's formula in \cite[Lemma IV.45.9]{Rogers_Williams_1987}, 
    \begin{align*}
        V(t) &- V(0) 
        \\
        &= \int_0^t \left(S^{-1}\left(\widetilde W(u)\right)\right)^{\frac{2a}{\sigma^2}} e^{- \frac{2bS^{-1}\left(\widetilde W(u)\right)}{\sigma^2}} d\widetilde W(u)
        \\
        &\quad + \int_0^t \frac{1}{\sigma^2} \left(S^{-1}\left(\widetilde W(u)\right)\right)^{\frac{4a}{\sigma^2} - 1} e^{-\frac{4b}{\sigma^2}S^{-1}\left(\widetilde W(u)\right)} \left(a-b S^{-1}\left(\widetilde W(u)\right)\right)du
        \\
        & = \int_0^t \left(V(u)\right)^{\frac{2a}{\sigma^2}} e^{-\frac{2b}{\sigma^2}V(u)}dZ(u) 
        \\
        &\quad + \int_0^t \frac{1}{\sigma^2} \left(V(u)\right)^{\frac{4a}{\sigma^2}-1} e^{-\frac{4b}{\sigma^2}V(u)} \left(a - b V(u)\right)du.
    \end{align*}
    Hence, by It\^o's formula, for any infinitely differentiable function with compact support $h$,
    \begin{equation}\label{proofeq: C is a local martingale}
        C_t(h) := h(V(t)) - h(V(0)) - \int_0^t \frac{1}{\sigma^2} \left(V(s)\right)^{\frac{4a}{\sigma^2} - 1} e^{-\frac{4b}{\sigma^2} V(s)} \mathcal A h(V(s)) ds
    \end{equation}
    is a local martingale, where
    \[
        \mathcal Ah(x) := (a-bx) h'(x) + \frac{\sigma^2x}{2} h''(x)
    \]
    is the generator of \eqref{0.q}. Recall that 
    \[
        X(t) = V(\tau_t)
    \]
    and observe that, by \eqref{proofeq: C is a local martingale}, for any infinitely differentiable function with compact support $h$, simple change of variables yields that
    \begin{align*}
        C_{\tau_t}(h) & = h(V(\tau_t)) - h(V(0)) - \int_0^{\tau_t} \frac{1}{\sigma^2} \left(V(s)\right)^{\frac{4a}{\sigma^2} - 1} e^{-\frac{4b}{\sigma^2} V(s)} \mathcal A h(V(s)) ds 
        \\
        & = h(X(t)) - h(x_0) - \int_0^t \mathcal A h(X(s)) ds.
    \end{align*}
    Since $C_{\tau_t}(h)$, $t\ge 0$, is a local martingale by the optional stopping theorem,
    \[
        h(X(t)) - h(x_0) - \int_0^t \mathcal A h(X(s)) ds
    \]
    is also a local martingale and therefore, by \cite[V.19--V.20]{Rogers_Williams_1987}, $X$ is the weak solution to \eqref{0.q}.

    \textbf{Step 2.} Let now $X$ be the unique strong solution to \eqref{0.q}. By Remark \ref{rem: inverse transformation} and Step 1, the process
    \[
        \widetilde W(t) := S(X(\varphi_t)),
    \]
    where $\varphi$ is defined as the inverse of the mapping $t \mapsto \int_0^t \frac{1}{\rho\left(X(s)\right)}ds$, is a reflected Brownian motion for which $X$ admits the representation \eqref{eq: X as time-changed Brownian motion}. 
    
\end{proof}

\subsection{Characterization of $L$ in terms of $L^{\widetilde W}$}

Having the representation \eqref{eq: X as time-changed Brownian motion} at our disposal, we are now ready to characterize the process $L$ from \eqref{eq: def of L} in terms of the local time $L^{\widetilde W}$ of the corresponding reflected Brownian motion. 

Let $X$ be the unique strong solution to the SDE \eqref{0.q} with $0< a < \frac{\sigma^2}{4}$ and $\widetilde W$ be the reflected Brownian motion such that
\[
     X(t) := S^{-1}\left( \widetilde W(\tau_t) \right).
\]
Denote $L^{\widetilde W} = \{L^{\widetilde W}(t,x),~t\ge 0,~x\ge 0\}$ the jointly continuous modification of the local time of $\widetilde W$ so that for any bounded measurable $f$,
\[
    \int_0^t f(\widetilde W(s))ds = \int_0^\infty f(x)L^{\widetilde W}(t,x)dx, \quad t\ge 0,
\]
with probability 1.

First of all, let us express the local time $L^Y = L^Y(t,y)$ of the process $Y = \sqrt{X}$ in terms of $L^{\widetilde W}$. 
\begin{proposition}
    Let
    \[
        Y(t) := \sqrt{X(t)} = \sqrt{S^{-1}\left( \widetilde W(\tau_t) \right)}
    \]
    be the square root of the CIR process $X$. Then, for any bounded measurable $f$,
    \[
        \int_0^t f(Y(s))ds = \int_0^\infty f(y) L^Y(t,y)dy,
    \]
    where, with probability 1,
    \begin{equation}\label{eq: LY represented as transf of LW}
        L^Y(t,y) =  \frac{2}{\sigma^2}  y^{\frac{4a}{\sigma^2} - 1} e^{-\frac{2b}{\sigma^2}y^2}  L^{\widetilde W}(\varphi_t, S(y^2))
    \end{equation}
    with $\varphi$ being defined by \eqref{eq: phi}.
\end{proposition}

\begin{proof}
    For any bounded measurable $f$, we can write 
    \begin{align*}
        \int_0^t f(Y(s))ds &= \int_0^t f\left( \sqrt{S^{-1}\left(\widetilde W(\tau_u)\right)} \right)du
        \\
        & = \int_0^{\varphi_t} f\left( \sqrt{S^{-1}\left( \widetilde W(z) \right)} \right) \rho\left(S^{-1}\left(\widetilde W(z)\right)\right)dz
        \\
        & = \int_0^{\infty} f\left( \sqrt{S^{-1}\left( x \right)} \right) \rho\left(S^{-1}\left(x\right)\right) L^{\widetilde W}\left(\varphi_t, x\right)dx
        \\
        & = \int_0^{\infty} f\left( y \right) \rho\left(y^2\right) L^{\widetilde W}\left(\varphi_t, S(y^2)\right) 2yS'(y^2) dy
        \\
        & =: \int_0^\infty f(y) L^Y(t,y)dy.
    \end{align*}
    The final result is obtained by recalling that
    \begin{equation*}
    \begin{aligned}
        L^Y(t,y) &=  \rho\left(y^2\right) L^{\widetilde W}\left(\varphi_t, S(y^2)\right) 2yS'(y^2)
        \\
        & = \frac{2}{\sigma^2} y^{\frac{4a}{\sigma^2} - 1} e^{-\frac{2b}{\sigma^2}y^2} L^{\widetilde W} (\varphi_t, S(y^2)).
    \end{aligned}    
    \end{equation*}
\end{proof}
Define a normalized local time of the process $Y$ as follows. Set
    \begin{equation}\label{eq: definition of ell}
        \ell(t,y) := y^{1-\frac{4a}{\sigma^2}} L^Y(t,y), \quad y>0,
    \end{equation}
   and 
\[\ell(t,0):=\lim_{y\to0+}\ell(t,y).\]
Note that $\ell(t,y)$ is continuous in $(t,y)$ because 
$$\ell(t,y)= \frac{2}{\sigma^2} e^{-\frac{2b}{\sigma^2}y^2}  L^{\widetilde W}(\varphi_t, S(y^2)).$$ However, we want to stress that $\ell(t,y)$ is a function of the local time $L^Y$ of the process $Y$ without mentioning the auxiliary Brownian motion $\widetilde W$.

\begin{theorem}\label{th: main theorem}
    Let $X$ be the CIR process satisfying \eqref{0.q} and $\widetilde W$ be the reflected Brownian motion such that $X(t) = S^{-1}(\widetilde W(\tau_t))$, $t\ge 0$. Then, with probability 1, the process $Y = \sqrt{X}$ satisfies the SDE of the form
    \begin{equation}\label{eq: dynamics of Y in main theorem}
        Y(t) = \sqrt{x_0} - \frac{b}{2}\int_0^t Y(s) ds + \frac{\sigma}{2} W(t) + L(t),
    \end{equation}
    where
    
    \begin{equation}\label{eq: L as limit and local time}
    \begin{aligned}
        L(t) = -\frac{1}{2}\left(\frac{\sigma^2}{4} - a \right) \int_0^{\infty} y^{\frac{4a}{\sigma^2} - 2} \left( \ell(t,y) - \ell(t,0) 
        \right) dy.
    \end{aligned}
    \end{equation}
    Moreover, 
    \[
   \begin{aligned}
        L(t) &:= \lim_{\varepsilon \downarrow 0} \frac{1}{2} \int_0^t \left( \frac{a}{\sqrt{X(s) + \varepsilon}} - \frac{\sigma^2}{4} \frac{X(s)}{(X(s) + \varepsilon)^{\frac{3}{2}}} \right)ds
        \\
        & = - \lim_{\varepsilon \downarrow 0} \frac{1}{2} \int_0^t \left( \frac{\frac{\sigma^2}{4} - a}{\sqrt{X(s)+\varepsilon}} - \frac{\sigma^2}{4} \frac{\varepsilon}{(X(s) + \varepsilon)^{\frac{3}{2}}} \right)ds.
    \end{aligned}
    \]
        
\end{theorem}

\begin{remark}
    Since the SDE \eqref{0.q} has a strong solution, Theorem \ref{th: main theorem} immediately yields that the SDE \eqref{eq: dynamics of Y in main theorem}--\eqref{eq: L as limit and local time} also has a strong solution.
\end{remark}

%\begin{remark}
%    By \eqref{eq: LY represented as transf of LW} and \eqref{eq: definition of ell}, $\ell$ can be expressed via the local time $L^Y$ of $Y$ as
%    \[
%        \ell(t,y) = y^{1-\frac{4a}{\sigma^2}} L^Y(t,y),
%    \]
%    with $\ell(t,0) = \lim_{y\to 0+} \ell(t,y)$.
%\end{remark}

\begin{remark}\label{rem: integral in L is good}
    Despite the fact that $\frac{4a}{\sigma^2} - 2 \in \left(-2,-1\right)$, the integral
    \[
        \int_0^{\infty} y^{\frac{4a}{\sigma^2} - 2} \left| \ell(t,y) - \ell(t,0)\right| dy
    \]
    is finite with probability 1. Indeed, denote $k := \frac{4a}{\sigma^2} \in(0,1)$ and observe that, by \eqref{eq: definition of ell} and properties of local time $L^{\widetilde W}$,
    \begin{equation}\label{eq: bound for ly minus l0}
        \sup_{y\ge 1} |\ell(t,y) - \ell(t,0)| < \infty
    \end{equation}
    with probability 1 for any $t\ge 0$. Moreover, since $L^{\widetilde W}(t,\cdot)$ is H\"older continuous of order up to $\frac{1}{2}$ a.s. (see e.g. calculations in \cite[Section IV.44]{Rogers_Williams_1987}), for any $\delta \in\left(0, \frac{1}{2}\right)$ and any fixed $t>0$ there exists a random variable $C>0$ such that, with probability~1,
    \begin{equation}\label{eq: Holder continuity of l}
        |\ell(t,y) - \ell(t,0)| \le C \cdot (S(y^2))^{\frac{1}{2}-\delta}.
    \end{equation}
    Hence, on the one hand,
    \[
        \int_1^\infty \left|\ell(t, y) - \ell(t, 0) \right|y^{k-2} dy < \infty \ \mbox{a.s.} 
    \]
    by \eqref{eq: bound for ly minus l0}. On the other hand, take $\delta \in \left(0,\frac{k}{2(2-k)}\right)$ and observe that \eqref{eq: Holder continuity of l} implies 
    \begin{align*}
        \int_0^1 \left|\ell(t, y) - \ell(t, 0) \right|y^{k-2}dy & \le  C \int_0^1 \left(S(y^2)\right)^{\frac{1}{2}-\delta} y^{k-2}dy 
        \\
        & = C \int_0^1 \left(\int_0^{y^2} z^{-\frac{k}{2}} e^{\beta z} dz\right)^{\frac{1}{2}-\delta} y^{k-2}dy 
        \\
        & \le C \int_0^1 \left(\int_0^{y^2} z^{-\frac{k}{2}} dz\right)^{\frac{1}{2}-\delta} y^{k-2}dy
        \\
        & \le C \int_0^1 y^{-1 + \frac{k}{2} - \delta'} dy < \infty \ \mbox{a.s.},
    \end{align*}
    where $\delta' := (2-k) \delta \in \left(0, \frac{k}{2}\right)$ and $C$ is a (random) constant that varies from line to line.
\end{remark}
Now we are ready to proceed to the proof of Theorem \ref{th: main theorem}.

\begin{proof}[Proof of Theorem \ref{th: main theorem}]
    In Section \ref{sec: preliminaries}, we obtained the representation \eqref{eq: dynamics of Y in main theorem} with $L$ being a \emph{ucp-limit} of the form
    \begin{align*}
        L(t) &= \lim_{\varepsilon \downarrow 0} \frac{1}{2} \int_0^t \left( \frac{a}{\sqrt{X(s) + \varepsilon}} - \frac{\sigma^2}{4} \frac{X(s)}{(X(s) + \varepsilon)^{\frac{3}{2}}} \right)ds
        \\
        & = - \lim_{\varepsilon \downarrow 0} \frac{1}{2} \int_0^t \left( \frac{\frac{\sigma^2}{4} - a}{\sqrt{X(s)+\varepsilon}} - \frac{\sigma^2}{4} \frac{\varepsilon}{(X(s) + \varepsilon)^{\frac{3}{2}}} \right)ds.
    \end{align*}
    Hence, one is left to prove that this limit exists in the sense of a.s. convergence and check that the last equality in \eqref{eq: L as limit and local time} holds.
    
    Let $k:= \frac{4a}{\sigma^2} \in(0,1)$ denote the dimension of the CIR process, i.e. we have to study the a.s.-limit of the form
    \begin{align*}
         L(t) &:= - \lim_{\varepsilon \downarrow 0} \frac{1}{2} \int_0^t \left( \frac{\frac{\sigma^2}{4} - a}{\sqrt{X(s)+\varepsilon}} - \frac{\sigma^2}{4} \frac{\varepsilon}{(X(s) + \varepsilon)^{\frac{3}{2}}} \right)ds
         \\
         & = - \frac{\sigma^2}{8} \lim_{\varepsilon \downarrow 0} \int_0^t \left( \frac{1-k}{\sqrt{X(s)+\varepsilon}} - \frac{\varepsilon}{\left(X(s) + \varepsilon \right)^{\frac{3}{2}}} \right)ds.
    \end{align*}
    Observe that
    \begin{equation*}
    \begin{aligned}
        \int_0^t \frac{1-k}{\sqrt{X(s)+\varepsilon}} ds &= \int_0^t \frac{1-k}{\sqrt{Y^2(s) + \varepsilon}} ds   = \int_0^\infty \frac{1-k}{\sqrt{y^2 + \varepsilon}}  L^Y(t,y) dy
        \\
        &= \int_0^\infty \frac{1-k}{\sqrt{y^2 + \varepsilon}}  y^{k-1}  \ell(t, y) dy
        \\
        & = \int_0^\infty \frac{1-k}{\sqrt{y^2 + \varepsilon}}  y^{k-1}  (\ell(t, y) - \ell(t, 0)) dy 
        \\
        &\quad + \ell(t, 0) \int_0^\infty \frac{1-k}{\sqrt{y^2 + \varepsilon}}  y^{k-1}  dy
    \end{aligned}    
    \end{equation*}
    and, similarly,
    \begin{equation*}
    \begin{aligned}
        \int_0^t \frac{\varepsilon}{\left(X(s) + \varepsilon \right)^{\frac{3}{2}}} ds & = \int_0^t \frac{\varepsilon}{(Y^2(s) + \varepsilon)^{\frac{3}{2}}} ds  = \int_0^\infty \frac{\varepsilon}{(y^2 + \varepsilon)^{\frac{3}{2}}} L^Y(t,y) dy
        \\
        &= \int_0^\infty \frac{\varepsilon}{(y^2 + \varepsilon)^{\frac{3}{2}}}  y^{k-1} \ell(t, y) dy
        \\
        &= \int_0^\infty \frac{\varepsilon}{(y^2 + \varepsilon)^{\frac{3}{2}}}  y^{k-1} (\ell(t, y) - \ell(t,0)) dy 
        \\
        &\quad+ \ell(t, 0) \int_0^\infty \frac{\varepsilon}{(y^2 + \varepsilon)^{\frac{3}{2}}} y^{k-1} dy.
    \end{aligned}    
    \end{equation*}
    Let us study separately the asymptotics of
    \begin{align*}
        I_1(\varepsilon) &:= \int_0^\infty \frac{1-k}{\sqrt{y^2 + \varepsilon}}  y^{k-1} (\ell(t, y) - \ell(t, 0)) dy,
        \\
        I_2(\varepsilon) & := \int_0^\infty \frac{\varepsilon}{(y^2 + \varepsilon)^{\frac{3}{2}}}  y^{k-1} (\ell(t, y) - \ell(t,0)) dy,
        \\
        I_3(\varepsilon) &:= \ell(t, 0) \left(\int_0^\infty \frac{1-k}{\sqrt{y^2 + \varepsilon}}  y^{k-1} dy - \int_0^\infty \frac{\varepsilon}{(y^2 + \varepsilon)^{\frac{3}{2}}} y^{k-1}  dy\right)
    \end{align*}
    as $\varepsilon \downarrow 0$. First, observe that for any $y \ge 0$
    \[
        \frac{1-k}{\sqrt{y^2 + \varepsilon}}  y^{k-1}  |\ell(t, y) - \ell(t, 0)| \le y^{k-2} |\ell(t, y) - \ell(t, 0)|
    \]
    and note that by Remark \ref{rem: integral in L is good},
    \[
       \int_0^{\infty} y^{k - 2} \left| \ell(t,y) - \ell(t,0)\right| dy < \infty \quad \mbox{a.s.}
    \]
    Thus, by the dominated convergence theorem, with probability 1,
    \begin{align*}
        \lim_{\varepsilon \downarrow 0} I_1(\varepsilon) &= \lim_{\varepsilon\downarrow 0}\int_0^\infty \frac{1-k}{\sqrt{y^2 + \varepsilon}} y^{k-1}  \left(\ell(t, y) - \ell(t, 0) \right) dy 
        \\
        &= (1-k) \int_0^\infty y^{k-2} \left(\ell(t, y) - \ell(t, 0) \right) dy.
    \end{align*}

    Next, observe that, with probability 1,
    \begin{equation}\label{I2-1}
    \begin{aligned}
        \bigg|\int_1^\infty &\frac{\varepsilon}{(y^2 + \varepsilon)^{\frac{3}{2}}} y^{k-1}  \left(\ell(t, y) - \ell(t, 0) \right) dy \bigg| 
        \\
        & \le \varepsilon \int_1^\infty \left|\ell(t, y) - \ell(t, 0) \right|y^{k-4} dy 
        \\
        &\le \varepsilon C \int_1^\infty y^{k-4}dy \to 0, \quad \varepsilon \downarrow 0.
    \end{aligned}   
    \end{equation}
    On the other hand, take an arbitrary $\delta\in \left(0,\frac{k}{2(2-k)}\right)$, denote $\delta' := (2-k)\delta$ and observe that \eqref{eq: Holder continuity of l} yields
    \begin{equation}\label{proofeq: I2-1}
    \begin{aligned}
        \int_0^1 &\frac{\varepsilon}{(y^2 + \varepsilon)^{\frac{3}{2}}} y^{k-1}  \left|\ell(t,y) - \ell(t, 0) \right| dy 
        \\
        & \le C \int_0^1 \frac{\varepsilon}{(y^2 + \varepsilon)^{\frac{3}{2}}} y^{k-1} (S(y^2))^{\frac{1}{2}-\delta} dy
        \\
        & = C \int_0^1 \frac{\varepsilon}{(y^2 + \varepsilon)^{\frac{3}{2}}} y^{k-1} \left(\int_0^{y^2} z^{-\frac{k}{2}} e^{\beta z} dz\right)^{\frac{1}{2}-\delta} dy
        \\
        & \le C \int_0^1 \frac{\varepsilon}{(y^2 + \varepsilon)^{\frac{3}{2}}} y^{k-1} \left(\int_0^{y^2} z^{-\frac{k}{2}}   dz\right)^{\frac{1}{2}-\delta} dy
        \\
        & \le C \int_0^1 \frac{\varepsilon}{(y^2 + \varepsilon)^{\frac{3}{2}}} y^{\frac{k}{2}-\delta'} dy
        \\
        & = C \varepsilon^{\frac{k}{4} -\frac{\delta'}{2}} \int_0^1 \frac{\varepsilon^{- \frac{1}{2}}}{ \left( \left({y}/{\sqrt{\varepsilon}}\right)^2 + 1 \right)^{\frac{3}{2}}} \left(\frac{y}{\sqrt{\varepsilon}}\right)^{\frac{k}{2} - \delta'} dy,
    \end{aligned}   
    \end{equation}
    where $\beta := \frac{2b}{\sigma^2}$. Hence, by substituting $z = y/\sqrt{\varepsilon}$ in \eqref{proofeq: I2-1}, we can write
    \begin{equation}\label{I2-2}
    \begin{aligned}
        \int_0^1 &\frac{\varepsilon}{(y^2 + \varepsilon)^{\frac{3}{2}}} y^{k-1}  \left|\ell(t,y) - \ell(t, 0) \right| dy 
%        \\
 %       & \le  C \varepsilon^{-\frac{\delta}{2}} \int_0^{\frac{1}{\sqrt{\varepsilon}}} \frac{\varepsilon^{\frac{k}{4}}}{ \left( z^2 + 1 \right)^{\frac{3}{2}}} z^{\frac{k}{2} - \delta} dz
        \\
        & \le  C \varepsilon^{\frac{k}{4}-\frac{\delta'}{2}} \int_0^{\infty} \frac{1}{ \left( z^2 + 1 \right)^{\frac{3}{2}}} z^{\frac{k}{2} - \delta'} dz
        \\
        & \to 0
    \end{aligned}    
    \end{equation}
    with probability 1 as $\varepsilon\downarrow 0$. Summarizing \eqref{I2-1} and \eqref{I2-2}, we obtain that, with probability 1,
    \begin{equation*}
    \begin{aligned}
        \lim_{\varepsilon \downarrow 0} I_2(\varepsilon) = 0.
    \end{aligned}
    \end{equation*}
    
    Finally, integration by parts yields
    \[
        \int_0^\infty \frac{y^{k-1}}{\sqrt{y^2+\varepsilon}}dy = \frac{1}{k} \int_0^\infty \frac{y^{k+1}}{(y^2+\varepsilon)^{\frac{3}{2}}} dy
    \]
    and the right-hand side of the last equation is equal to
    \begin{align*}
            & \frac{1}{k} \int_0^\infty \frac{y^{k-1} (y^2 + \varepsilon - \varepsilon)}{(y^2+\varepsilon)^{\frac{3}{2}}} dy 
        \\
        &= \frac{1}{k} \int_0^\infty \frac{y^{k-1}}{\sqrt{y^2+\varepsilon}} dy - \frac{1}{k} \int_0^\infty \frac{\varepsilon y^{k-1}}{(y^2+\varepsilon)^{\frac{3}{2}}} dy.
    \end{align*}
    Therefore
    \[
        \int_0^\infty \frac{\varepsilon}{(y^2 + \varepsilon)^{\frac{3}{2}}} y^{k-1} dy = \int_0^\infty \frac{1 - k}{\sqrt{y^2 + \varepsilon}} y^{k-1} dy 
    \]
    and
    \[
        I_3(\varepsilon) = \ell(t, 0) \left( \int_0^\infty \frac{1 - k}{\sqrt{y^2 + \varepsilon}} y^{k-1} dy -   \int_0^\infty \frac{\varepsilon}{(y^2 + \varepsilon)^{\frac{3}{2}}} y^{k-1} dy\right) = 0.
    \]
    Summarizing all of the above and recalling that $k=\frac{4a}{\sigma^2}$, we finally obtain that with probability 1
    \begin{align*}
        L(t) &= -\frac{\sigma^2}{8} \lim_{\varepsilon \downarrow 0} \left( I_1(\varepsilon) - I_2(\varepsilon) + I_3(\varepsilon)\right)
        \\
        & = -\frac{\sigma^2}{8} \left(1-k\right) \int_0^\infty y^{k-2} \left(\ell(t, y) - \ell(t, 0) \right) dy
        \\
        & = -\frac{1}{2}\left(\frac{\sigma^2}{4} - a \right) \int_0^{\infty} y^{\frac{4a}{\sigma^2} - 2} \left( \ell(t,y) - \ell(t,0) 
        \right) dy 
    \end{align*}
    which ends the proof.
\end{proof}

\begin{remark}\label{rem: a.s. convergence}
    Theorem \ref{th: main theorem} implies that the limit 
    \begin{equation}\label{eq: L a.s. convergence}
        L(t) = \lim_{n \to\infty} \frac{1}{2} \int_0^t \left( \frac{a}{\sqrt{X(s) + \varepsilon_n}} - \frac{\sigma^2}{4} \frac{X(s)}{(X(s) + \varepsilon_n)^{\frac{3}{2}}} \right)ds
    \end{equation}
    exists a.s. for any sequence $\{\varepsilon_n,~n\ge 1\}$ such that $\varepsilon_n\downarrow 0$ and does not depend on the particular choice of the sequence. Moreover, the proof of Theorem \ref{th: main theorem} yields that the existence of the limit \eqref{eq: L a.s. convergence} is ensured for all $\omega$ such that $L^{\widetilde W}(\omega; t, \cdot)$ is H\"older continuous. In other words, the set of full probability where \eqref{eq: L a.s. convergence} holds can be chosen independently of a particular sequence $\{\varepsilon_n,~n\ge 1\}$, as anticipated in Remark \ref{rem: anticipation of a.s. convergence}.   
\end{remark}

\section{Discussion of the results}\label{sec: simple cases}

It is evident that the nature of the limit in \eqref{eq: def of L} heavily depends on the relation between parameters $a$ and $\sigma$. Therefore, in order to put our findings from Section \ref{sec: hard case} into context, let us provide some relevant results from \cite{MYT2022} on the behavior of $Y$ when $a\ge\frac{\sigma^2}{4}$.

\subsection{Square root of the CIR process when $a\ge \frac{\sigma^2}{4}$}

\paragraph{Case I: $a>\frac{\sigma^2}{4}$.} Observe that, if 
\begin{equation}\label{eq: int 1 over Y}
    \int_0^t \frac{1}{\sqrt{X(s)}}ds = \int_0^t \frac{1}{Y(s)}ds < \infty \quad a.s.,
\end{equation}
then the limit \eqref{eq: def of L} is equal to
\[
    L(t) = \frac{1}{2} \left(a - \frac{\sigma^2}{4}\right) \int_0^t \frac{1}{Y(s)}ds
\]
by monotone convergence. This is clearly the case for $a \ge \frac{\sigma^2}{2}$: indeed $a \ge \frac{\sigma^2}{2}$ implies that $X$ (and hence $Y$) has strictly positive paths a.s. (see e.g. \cite{Andersen_Piterbarg_2005} or \cite{Cox_Ross_1976}) and therefore \eqref{eq: int 1 over Y} holds for all $t\ge 0$. It turns out (see e.g. \cite[Theorem 2.1(a)]{MYT2022}) that \eqref{eq: int 1 over Y} also holds if $\frac{\sigma^2}{4}<a< \frac{\sigma^2}{2}$, i.e. one can prove the following result.

\begin{theorem}\label{th: greater case}{\emph{(\cite[Theorem 2.1(a)]{MYT2022})}.} Let $a>\frac{\sigma^2}{4}$. Then, for any $t\ge 0$,
    \begin{equation}\label{eq: int 1 over Y 2}
        \int_0^t \frac{1}{Y(s)}ds < \infty \quad a.s.
    \end{equation}
    and $Y$ a.s. satisfies the SDE of the form
    \begin{equation}\label{equ:posit}
        Y(t) = \sqrt{x_0} + \frac{1}{2} \left(a - \frac{\sigma^2}{4}\right) \int_0^t \frac{1}{Y(s)}ds - \frac{b}{2} \int_0^t Y(s) ds + \frac{\sigma}{2} W(t). 
    \end{equation}
\end{theorem}

\begin{remark}
    Using the same arguments as in \cite[Theorem 3.2]{Cherny_2000}, it is possible to prove that for $a > \frac{\sigma^2}{4}$ the process $Y = \sqrt{X}$ is the unique \emph{non-negative} strong solution to the SDE \eqref{equ:posit}. However, if $\frac{\sigma^2}{4} < a < \frac{\sigma^2}{2}$, \eqref{equ:posit} has other strong solutions; moreover, the uniqueness in law does not hold for \eqref{equ:posit}. For a more detailed discussion of this phenomenon, we refer the reader to \cite{Blei_2012} whereas a comprehensive overview of SDEs of the type \eqref{equ:posit} can be found in \cite{Cherny_Engelbert_2005}. 
\end{remark}

\paragraph{Case II: ${a = \frac{\sigma^2}{4}}$.} The case $a = \frac{\sigma^2}{4}$ turns out to be different from the one described above: in this regime, $X$ can hit zero (see e.g. \cite{Andersen_Piterbarg_2005} or \cite{Cox_Ross_1976}) and, as noted in e.g. \cite[Theorem 2.1(b)]{MYT2022}, \eqref{eq: int 1 over Y 2} does not hold for all $$t>\inf\{s\ge 0~|~Y(s) = 0\}.$$ However, the limit $L$ from \eqref{eq: def of L} has a simple interpretation in terms of Skorokhod reflections (see e.g. the seminal works \cite{Skorokhod1961, Skorokhod1962}) as summarized in the following theorem.

\begin{theorem}\label{th: equal case}{\emph{(\cite[Theorem 2.1(b)]{MYT2022})}.} Let $a=\frac{\sigma^2}{4}$ and denote $\tau:= \inf\{s\ge 0~|~X(s) = 0\}$.
    \begin{itemize}
        \item[1)] For all $\gamma>0$,
        \[
            \int_0^{\tau+\gamma} \frac{1}{Y(s)}ds = \infty \quad a.s.
        \]

        \item[2)] The processes $Y:= \sqrt{X}$ and $L$ defined by \eqref{eq: def of L} is the (unique) solution to Skorokhod problem
        \begin{equation}\label{eq: representation of Y, reflected case}
            Y(t) = \sqrt{x_0} - \frac{b}{2} \int_0^t Y(s)ds + \frac{\sigma}{2}W(t) + L(t),
        \end{equation}
        with $L$ being the corresponding Skorokhod reflection function, i.e. a continuous non-decreasing process starting at 0 with points of growth occurring only at zeros of $Y$ and such that $Y(t) \ge 0$.
    \end{itemize}
\end{theorem} 

\begin{remark}
    Item 2) of Theorem \ref{th: equal case} states that,  when $a = \frac{\sigma^2}{4}$, the square root process $Y = \sqrt{X}$ coincides with a \textit{reflected Ornstein-Uhlenbeck  (ROU) process}. More details on the latter can be found in e.g. \cite{WG2003-2}.
\end{remark}

\subsection{Comparison to the low-dimensional case}

As we have seen in Section \ref{sec: hard case}, the case $0 < a < \frac{\sigma^2}{4}$ is arguably the most challenging one and leads to the most involved value of the limit \eqref{eq: def of L}. First of all, note that \eqref{eq: int 1 over Y} does not hold due to Theorem~\ref{th: equal case} together with the comparison theorem for solutions of SDEs (see e.g. \cite{IW1977}). Next, the limit $L$ in \eqref{eq: def of L} cannot be non-decreasing in $t$ as it happens when $a\ge \frac{\sigma^2}{4}$. Indeed, consider $\tau \ge 0$ such that $X(\tau) > 0$. Then, by a.s. continuity of $X$, there exists a neighborhood $\tau_- < \tau < \tau_+$ such that $X$ is bounded away from zero on $(\tau_-,\tau_+)$. Denote now $-\delta := a - \frac{\sigma^2}{4}$, $\delta>0$. Then, with probability 1, for all $\tau_- < t_1 < t_2 < \tau_+$
    \begin{align*}
        L(t_2) - L(t_1) & = \lim_{n\to\infty} \frac{1}{2} \int_{t_1}^{t_2} \left( \frac{a}{\sqrt{X(s) + \varepsilon_n}} - \frac{\sigma^2}{4} \frac{X(s)}{(X(s) + \varepsilon_n)^{\frac{3}{2}}} \right)ds
        \\
        &= \lim_{n\to\infty} \frac{1}{2} \int_{t_1}^{t_2} \left( \frac{\sigma^2}{4}\frac{\varepsilon_n}{(X(s) + \varepsilon_n)^{\frac{3}{2}}} -  \frac{\delta}{\sqrt{X(s) + \varepsilon_n}} \right)ds
        \\
        & =  -\frac{1}{2} \int_{t_1}^{t_2} \frac{\delta}{\sqrt{X(s)}} ds 
        \\
        &< 0.
    \end{align*}
On the other hand, $L$ is not strictly decreasing on the entire $[0,T]$: if it is strictly decreasing (and, since $L(0)=0$, non-positive), then $Y \le U$, where $U$ is the standard Ornstein-Uhlenbeck process defined by
\[
    U(t) = Y(0) - \frac{b}{2}\int_0^t U(s) ds + \frac{\sigma}{2} W(t).
\]
However, it is not possible since $Y$ cannot take negative values.

\section{Connection to Skorokhod reflections}\label{sec: connection to Skorokhod reflections}

Finally, let us present the connection of low-dimensional CIR processes with Skorokhod problems. For $\delta > -\frac{\sigma^2}{4}$, consider a family of CIR processes
$\{X_\delta\}$ with $a = a(\delta) = \frac{\sigma^2}{4} + \delta$ and defined by
\begin{equation}\label{eq: family of CIR}
    X_\delta(t) = x_0 + \int_0^t \left(\frac{\sigma^2}{4} + \delta - bX_\delta(s)\right)ds + \sigma \int_0^t \sqrt{X_\delta(s)}dW(s).
\end{equation}
As described above in Sections \ref{sec: hard case}--\ref{sec: simple cases}, the process $Y_\delta := \sqrt{X_\delta}$ satisfies the SDE of the form
\[
    Y_\delta(t) = \sqrt{x_0} - \frac{b}{2} \int_0^t Y_\delta(s) ds + \frac{\sigma}{2} W(t) + L_{\delta}(t),
\]
where the term $L_\delta$ depends on the parameter $\delta$ as follows:
\begin{itemize}
    \item[---] if $-\frac{\sigma^2}{4} < \delta < 0$,
    \begin{align*}
        L_{\delta}(t) &= -\frac{1}{2}\left(\frac{\sigma^2}{4} - a \right) \int_0^{\infty} y^{\frac{4a}{\sigma^2} - 2} \left( \ell(t,y) - \ell(t,0) 
        \right) dy
        \\
        & =  \frac{\delta}{2}\int_0^{\infty} y^{\frac{4\delta}{\sigma^2} - 1} \left( \ell(t,y) - \ell(t,0) 
        \right) dy,
    \end{align*}
    where $\ell(t,y)=y^{-\frac{4\delta}{\sigma^2}} L^{Y_\delta}(t,y)$ is the normalized local time of $Y_{\delta}$, see  \eqref{eq: definition of ell};

    \item[---] if $\delta > 0$,
    \begin{equation*}\label{eq: sqrtCIR, delta positive}
        L_\delta(t) = \frac{1}{2} \int_0^t \frac{\delta}{Y_\delta(s)}ds
    \end{equation*}
    and the integral is well-defined and finite with probability 1;

    \item[---] if $\delta = 0$, $L_{0}$ is the Shorokhod reflection function, i.e. a continuous non-decreasing process with points of growth occurring only at zeros of $Y_0$ and such that $Y_{0} \ge 0$, which is a symmetric local time of $Y_0$ at 0; in particular, $Y_0$ is a reflected Ornstein-Uhlenbeck process.
\end{itemize}

The dynamics of $Y_{\delta}$ with $\delta \ge 0$ described above allowed \cite{MYT2022} to obtain the following alternative representation to the Skorokhod reflection function $L_0$.
\begin{theorem}\label{th: representation of L from the right}{\emph{(\cite[Theorem 2.4]{MYT2022})}.}
    Let $\{\delta_n,~n\ge 1\}$ be an arbitrary positive sequence such that $\delta_n \downarrow 0$, $n\to \infty$. Then, with probability 1, for any $T>0$
    \[
        \sup_{t\in[0,T]} |Y_{\delta_n}(t) - Y_0(t)| \to 0
    \]
    and
    \[
        \sup_{t\in[0,T]} \left| L_0(t) - L_{\delta_n}(t) \right| = \sup_{t\in[0,T]} \left| L_0(t) - \frac{1}{2} \int_0^t \frac{\delta_n}{\sqrt{X_{\delta_n}(s)}}ds \right| \to 0
    \]
    as $n\to \infty$.
\end{theorem}
Theorem \ref{th: representation of L from the right} essentially concerns the case $\delta \to 0+$ but does not discuss what happens when $\delta\to 0-$, so we finalize the Section by filling this gap. 

\begin{theorem}
    Let $\{\delta_n,~n\ge 1\}$ be an arbitrary positive sequence such that $\delta_n \downarrow 0$, $n\to\infty$. Then, with probability 1, for any $T>0$
    \begin{equation}\label{eq: a.s. convergence of Y from the left}
        \sup_{t\in[0,T]} |Y_{-\delta_n}(t) - Y_0(t)| \to 0
    \end{equation}
    and
    \begin{equation}\label{eq: a.s. convergence of L from the left}
       \sup_{t\in[0,T]} \left| L_{-\delta_n}(t) - L_0(t) \right| \to 0
    \end{equation}
    as $n\to \infty$.
\end{theorem}
\begin{proof}
    By \cite[Theorem 1.1]{IW1977}, for any $t\ge 0$, $X_{-\delta_n}(t) \le X_{-\delta_{n+1}}(t) \le X_0(t)$ a.s. Moreover, by \cite[Theorem 4.1]{Mishura2009},
    \[
        \sup_{t\in [0,T]} \mathbb E\left[\left|X_{-\delta_n}(t) - X_{0}(t)\right|\right] \to 0, \quad n \to \infty,
    \]
    and hence, for all $t\ge 0$,
    \[
        Y_{-\delta_n}(t) \xrightarrow{\mathbb P} Y_{0}(t), \quad n \to \infty,
    \]
    and
    \[
        \int_0^t Y_{-\delta_n}(s)ds \xrightarrow{\mathbb P} \int_0^t Y_{0}(s)ds, \quad n\to \infty.
    \]
    Therefore, since monotone convergence in probability implies almost sure convergence, for any $t\ge 0$
    \[
        Y_{-\delta_n}(t) \to Y_{0}(t)
    \]
    and
    \[
        \int_0^t Y_{-\delta_n}(s)ds \to \int_0^t Y_{0}(s)ds
    \]
    a.s. as $n \to \infty$ and hence, with probability 1,
    \begin{align*}
        L_{-\delta_n}(t) & = Y_{-\delta_n}(t) - \sqrt{x_0} + \frac{b}{2} \int_0^t Y_{-\delta_n} (s)ds - \frac{\sigma}{2} W(t)
        \\
        &\to  Y_{0}(t) - \sqrt{x_0} + \frac{b}{2} \int_0^t Y_{0} (s)ds - \frac{\sigma}{2} W(t)
        \\
        &= L_0(t), \quad n \to \infty.
    \end{align*}
    It remains to note that $Y_0$ as well as each $Y_{-\delta_n}$ have a.s. continuous paths and $\{Y_{-\delta_n}(t),~n\ge 1\}$ is non-decreasing a.s. w.r.t. $n$, which immediately yields \eqref{eq: a.s. convergence of Y from the left} by Dini's theorem. Similarly, $L_0$ as well as all $ L_{-\delta_n}$ are continuous with probability 1 and
    \begin{align*}
        L_{-\delta_n}(t) & = Y_{-\delta_n}(t) - \sqrt{x_0} + \frac{b}{2} \int_0^t Y_{-\delta_n} (s)ds - \frac{\sigma}{2} W(t)
        \\
        &\le Y_{-\delta_{n+1}}(t) - \sqrt{x_0} + \frac{b}{2} \int_0^t Y_{-\delta_{n+1}} (s)ds - \frac{\sigma}{2} W(t)
        \\
        & = L_{-\delta_{n+1}}(t),
    \end{align*}
    which implies \eqref{eq: a.s. convergence of L from the left}.
\end{proof}

\section{Acknowledgements}

The present research is carried out within the frame and support of the ToppForsk project nr. 274410 of the Research Council of Norway with title STORM: Stochastics for Time-Space Risk Models. The first  author is supported by The Swedish Foundation for Strategic Research, grant Nr. UKR22-0017.

\bibliographystyle{acm}
\bibliography{biblio}

\end{document}